\documentclass[footinclude=true]{amsart}
\usepackage{amsmath}
\usepackage{amsthm}
\usepackage{graphicx}
\usepackage{amssymb}
\usepackage{epstopdf}
\usepackage{nicefrac}
\usepackage{mathrsfs}
\usepackage{mathrsfs}
\usepackage{hyperref}
\usepackage{enumitem} 
\usepackage{subcaption}
\usepackage{tikz}
\usepackage{caption}
\usepackage{subcaption}

\theoremstyle{plain}

\newtheorem{theo}{Theorem}[section]
\newtheorem{lm}[theo]{Lemma}
\newtheorem*{q}{Question}

\newtheorem*{ack}{Ackowledgements}
\newtheorem*{notat}{Notation}

\newtheorem{cor}[theo]{Corollary}

\newtheorem*{example*}{Example}

\newtheorem*{Organisation}{Organisation}
\DeclareMathOperator{\Ha}{H} 
 
\DeclareMathOperator{\conv}{conv} 
\DeclareMathOperator{\PW}{PW}

\DeclareMathOperator{\supp}{supp} 
 
\DeclareMathOperator{\dist}{dist} 
\DeclareMathOperator{\sgn}{sgn} 
 
\DeclareMathOperator{\diag}{diag} 
\title{Nehari's Theorem and Hardy's inequality for Paley--Wiener spaces}
\author{Konstantinos Bampouras}
\address{Department of Mathematical Sciences, Norwegian University of Sciences and Technology (NTNU), NO-7491 Trondheim, Norway}
\email{konstantinos.bampouras@ntnu.no}

\date{\today}

\usepackage{cite}
\begin{document}
	
	
	\maketitle
	\begin{abstract}
		Recently it was proven that for a convex subset of $\mathbb{R}^{n}$ that has infinitely many extreme vectors, the Nehari theorem fails, that is, there exists a bounded Hankel operator $\Ha_{\phi}$ on the Paley--Wiener space $\PW(\Omega)$ that does not admit a bounded symbol. In this paper we examine whether Nehari's theorem can hold under the stronger assumption that the Hankel operator $\Ha_{\phi}$ is in the Schatten class $S^{p}(\PW(\Omega))$. We prove that this fails for $p>4$ for any convex subset of $\mathbb{R}^{n}$, $n\geq2$, of boundary with a $C^{2}$ neighborhood of nonzero curvature. Furthermore we prove that for a polytope $P$ in $\mathbb{R}^{n}$, the inequality $$\int_{2P}\dfrac{|\widehat{f}(x)|}{m(P\cap (x-P))}dx\leq C(P)\|f\|_{L^{1}},$$ holds for all $f\in \PW^{1}(2P)$, and consequently any Hilbert--Schmidt Hankel operator on a Paley--Wiener space of a polytope is generated by a bounded function.
	\end{abstract}
	\section{Introduction} 
	For a convex set $\Omega$ in $\mathbb{R}^{n}$, we define the Paley--Wiener space with respect to $\Omega$ to be the closed subspace of $L^{2}(\mathbb{R}^{n})$ with Fourier transform supported in $\Omega$, i.e. 
	$$\PW(\Omega)=\{f\in L^{2}(\mathbb{R}^{n}):\supp\widehat{f}\subset \Omega\},$$ where the Fourier is defined on $L^1(\mathbb{R}^{n})$ functions as $$\widehat{f}(x)=\int_{\mathbb{R}^n}f(\xi)e^{-2\pi i\langle x,\xi\rangle}d\xi,$$ and admits a well-known extension to the space of tempered distributions. Since the boundary of a convex set has measure zero, it makes no difference whether we consider convex sets to be open or closed. When $\Omega$ is bounded and convex, $\PW(\Omega)$ consists precisely of the Fourier transforms of entire functions on $\mathbb{C}^n$ of suitable exponential type, restricted to $\mathbb{R}^n$ see \cite{MR1451142}. We are interested in the study of Hankel operators defined on these spaces. The classical Hankel operators are operators on the Hardy space of the unit disc $H^{2}(\mathbb{T})$ with matrices of fixed anti-diagonals with respect to the basis $\{z^{d},d=0,1,2,...\}$. For information about Hardy spaces and classical Hankel operators the reader can refer to \cite{MR1892647}. These operators can be obtained by $L^{2}(\mathbb{T})$ functions, i.e for a Hankel operator $H$ there exists an $L^{2}(\mathbb{T})$ function $\phi$, not necessarily unique, which is called a symbol of $H$, such that the operator has the form $H(f)=P(\overline{f^\ast} \phi)$, where $P$ is the Riesz projection and $f^\ast(z)=\overline{f}(\overline{z})$. We note that $f^{\ast}$ can also be defined by the property $\widehat{f^\ast}=\overline{\widehat{f}},$ which will be helpful when we will define Hankel operators for Paley--Wiener spaces. The operator properties of Hankel operators have been well studied with respect to their symbol, for example compactness \cite{MR108684}, finite rank \cite{MR0237286} and Schatten classes \cite{MR1949210, MR674875}. A very famous result for these operators is a theorem of Z. Nehari \cite{MR82945}, that states that a Hankel operator $H:H^2(\mathbb{T})\to H^2(\mathbb{T})$ is bounded if and only if it admits a bounded symbol, i.e. there is a bounded function $\phi\in L^{\infty}(\mathbb{T})$ such that $H$ has matrix $(\widehat{\phi}(n+m))_{n,m\geq 0}$. More specifically, Nehari showed that there is a bounded function $\phi\in L^{\infty}(\mathbb{T})$ such that $\|H\|=\|\phi\|_{L^\infty}$.
	\par As in the classical case, we define the Hankel operators on $\PW(\Omega)$ as follows. Let $\widehat{\phi}$ be a distribution and let $\Ha_\phi$ be the densely defined operator given by 
	$$\widehat{\Ha_{\phi}(f)}(x)=\widehat{P_{\Omega}(\overline{f^{\ast}}\phi)}(x)=\int_{\Omega} \widehat{\phi}(x+y)\widehat{f}(y)dy, \quad x\in \Omega,$$
	where $P_{\Omega}$ is the projection of $L^{2}(\mathbb{R}^{n})$ onto $\PW(\Omega)$ and $f^{\ast}$ satisfies $\widehat{f^{\ast}}=\overline{\widehat{f}}$. It is evident that $\|\Ha_{\phi}\|\leq \|\phi\|_{L^{\infty}}$, thus every bounded function generates a bounded Hankel operator. Since the Hardy spaces of the disc and of the upper half plane are closely related \cite[Chapter 6]{MR1892647}, the results of the disc can be transferred to the real half line $\PW(\mathbb{R}_{+})$ and show that the analogue of Nehari's theorem holds. R. Rochberg \cite{MR878246} translated these results onto the real line segment $\PW(-1,1)$. However, O. F. Brevig and K.--M. Perfekt \cite{MR4502777} showed that for the disc the analogous result does not hold. It was later shown \cite{MR4700194} that the same can be said for all convex sets with infinitely many extreme vectors, and more specifically for all bounded convex sets that are not polytopes. Since the Schatten classes $S^{p}$ are an increasing family of spaces with respect to $p\geq 1$, and are contained in the space of bounded operators, we are interested in the following weaker version of Nehari's theorem, which was asked in \cite{MR4502777}. 
	\begin{q}
	For which values of $p\in [1,\infty)$ does every Hankel operator in the Schatten class $S^p(\PW(\Omega))$ admit a bounded symbol?
	\end{q}  
	We will say that Nehari's theorem for $\Omega$ holds (or fails) for some $p\in [1,\infty)$ if the question above is positive (or negative) for that $p$. If $\Omega$ is strictly convex set, then Nehari's theorem holds for $\Omega$ for $p=1$, \cite{bampouras2024besovspacesschattenclass}. For larger values of $p$ we will prove the following negative answer.
	
	\begin{theo}\label{ball}
		Suppose that $\Omega$ is a convex subset of $\mathbb{R}^{n}$, $n\geq 2$ that has a $C^{2}$ boundary neighborhood with nonzero curvature. Then Nehari's theorem fails for $\Omega$ for all $p>4$.
	\end{theo}
    By $C^{2}$ boundary neighborhood with nonzero curvature, we mean that there is an open ball $B$, such that $\partial \Omega\cap B$ is a $C^2$ manifold of strictly positive Gaussian. What we will actually need in our proof is the following: for each $x\in\partial \Omega\cap B$ we can find continuous radii $r_1(x),r_2(x)>0$ and balls $B_1(r_1(x)),B_2(r_2(x))$ containing $x$ at their boundary such that $B_1(r_1(x))\subset \Omega$ and $\partial\Omega\cap B\subset B_2(r_2(x))$.
    \par It is worth noting that the Schatten classes of Hankel operators on Paley--Wiener spaces have been studied by L. Peng \cite{MR953994, MR1001657}, and recently by the author of this work and K.--M. Perfekt \cite{bampouras2024besovspacesschattenclass}, and have been related to Besov spaces. It is not the object of study of this note.
	\par An interesting question is whether Nehari's theorem holds for $p=2$. In the case of the Hardy spaces in $\mathbb{T}^n$, it is not known whether Nehari's theorem holds, even though there have been noteworthy attempts \cite{MR1961195,MR1785283,MR766959}. In the Hardy space of infinite dimensions however, which are related to Dirichlet series of square summable coefficients, it was shown by J. Ortega--Cerd\`{a} and K. Seip \cite{MR2993031} that Nehari's theorem fails. For $p=2$, Nehari's theorem holds for infinite dimensions, and thus for all finite dimensions, as was shown by H. Helson \cite{MR2263964}. For more information about the infinite dimensional Hardy space the reader can refer to \cite[Chapter 4]{MR3099268}. O. F. Brevig and K.--M. Perfekt \cite{MR3422087} showed that for the infinite dimensional Hankel operators, Nehari's theorem fails for all $p>(1-\log\pi/\log 4)^{-1}= 5.7388...$.
	\par 
	Let us define the $\omega_{\Omega}$ function of the convex set $\Omega\subset\mathbb{R}^{n}$ as the convolution function of the characteristic of $\Omega$ with itself, 
	$$\omega_{\Omega}(x)=\chi_{\Omega}\ast\chi_{\Omega}(x)=m(\Omega\cap(x-\Omega)),$$ which is finite if and only if $\Omega$ does not contain lines. This function plays an important role in the understanding of the Schatten classes of Hankel operators on Paley--Wiener spaces (see \cite{bampouras2024besovspacesschattenclass} for further information) and characterizes the Hilbert--Schmidt Hankel operators in the following sense $$\|\Ha_{\phi}\|_{S^{2}}=\|\widehat{\phi}\sqrt{\omega_{\Omega}}\|_{L^{2}}.$$
	In the case of the classical Hardy space, the role of the $\omega$ function is played by the function $\omega(j)=(j+1)\chi_{\mathbb{Z}_{+}}(j),$ $j\in\mathbb{Z}$, since $\|\Ha_{\phi}\|_{S^{2}}=\|\sqrt{j+1}\widehat{\phi}(j)\|_{\ell^{2}(\mathbb{Z}_{+})}.$ 
	\par In the case of the infinite dimensional Hardy space, the proof of Helson's inequality is based on Carleman's inequality, $\sum_{j\geq 0}\frac{|\widehat{f}(j)|^2}{j+1}\leq C\|f\|_{L^{1}}^2$, for which the best constant is $C=1$. D. Vukoti\'{c} gave an elementary proof of this inequality for $C=1$, \cite{MR1984405}, based on the Riesz factorization of $H^1(\mathbb{T})$, which is stronger than Nehari's theorem. 
	\par An alternative proof for Carleman's inequality that does not make use of Nehari's theorem (though it gives a weaker constant $C=\pi$), is by using the fact that there exists a bounded function $\psi$, such that its positive Fourier coefficients are equal to $\widehat{\psi}(j)=\frac{1}{j+1}$, $j\geq 0$, namely $\psi(t)=i(\pi-t)e^{-it},$ $0<t<2\pi$. The analogue in the case of Paley--Wiener spaces would be the following. Let $\Omega$ be a convex bounded subset of $\mathbb{R}^n$ and suppose that for an increasing sequence of open and convex sets $K_{j}$, $j\geq 0$ with $\cup_{j}K_{j}=\Omega$, we can find bounded functions $\psi_{j}\in L^{\infty}(\mathbb{R}^{n})$ such that $\widehat{\psi}_j=\dfrac{1}{\omega_{\Omega}}$ as distributions in $2K_{j}$ and $\|\psi_{j}\|_{L^{\infty}}\leq C<\infty$. Although  $\widehat{\psi}_j$  may have nonzero values outside $2K_j$, this will prove to be irrelevant in the context of Paley--Wiener spaces, as we shall see. For $f\in \PW^1(2\Omega)=\{f\in L^{1}(\mathbb{R}^{n}):\supp\widehat{f}\subset 2\Omega\}$ let us set $\widehat{f}_r(x)=\widehat{f}(\frac{x-2(1-r)z}{r})$, where $z\in \Omega$ is fixed and $r\in (0,1)$. Since for every $r\in (0,1)$ there is $j_r\in\mathbb{Z}_{+}$ such that $r\Omega+(1-r)z\subset K_{j_r}$, Fatou's lemma and Young's inequality give
	\begin{eqnarray}
		\int_{2\Omega}\dfrac{|\widehat{f}(x)|^{2}}{\omega_{\Omega}(x)}dx&\leq& \liminf_{r\to 1}\int_{2\Omega}\dfrac{|\widehat{f}_r(x)|^{2}}{\omega_{\Omega}(x)}dx=\liminf_{r\to 1}\int_{\mathbb{R}^{n}}|\widehat{f}_r(x)|^{2}\widehat{\psi}_{j_r}(x)dx \nonumber \\
		&=&\liminf_{r\to 1}\langle f_r\ast\psi_{j_r},f_r\rangle\leq C\|f_r\|_{L^{1}}^2=C\|f\|_{L^{1}}^2, \nonumber 
	\end{eqnarray}
where the last equality holds since $\|f_r\|_{L^1}=\|f\|_{L^1}$.
	Therefore a duality argument would imply Nehari's theorem for $p=2$. The above reasoning for an arbitrary $\Omega$ seems implausible. However, in the case of a polytope we are able to give a stronger result (see Section \ref{sec3}) which we will see is a consequence, though not an immediate one, of the Nehari theorem for the Hardy space $H^2(\mathbb{T})$. We recall that a polytope in $\mathbb{R}^{n}$ is defined as a bounded finite intersection of half-spaces.
	We give the following result.
	\begin{theo}\label{poly}
		Let $P$ be a polytope in $\mathbb{R}^{n}$. Then there exists a constant $C(P)<\infty$ such that 
		$$\int_{2P}\dfrac{|\widehat{f}(x)|}{\omega_{P}(x)}dx\leq C(P)\|f\|_{L^{1}},$$ for all $f\in \PW^{1}(2P)=\{f\in L^{1}(\mathbb{R}^{n}):\supp\widehat{f}\subset 2P\}$. Furthermore, Nehari's theorem holds for $p=2$ for all polytopes.
	\end{theo}
    In the case where $P$ has empty interior, the integral on the left hand side is zero, so the inequality holds trivially. Therefore the interesting case is when $P$ has non-empty interior. 
    \par This inequality is stronger than the corresponding Carleman's inequality for Paley--Wiener spaces since 
    $$\int_{2P}\dfrac{|\widehat{f}(x)|^{2}}{\omega_{P}(x)}dx\leq \|\widehat{f}\|_{L^{\infty}}\int_{2P}\dfrac{|\widehat{f}(x)|}{\omega_{P}(x)}dx\leq \|f\|_{L^{1}}\int_{2P}\dfrac{|\widehat{f}(x)|}{\omega_{P}(x)}dx.$$
	In the setting of Hardy spaces, this inequality is the analogue of Hardy's inequality for $H^{1}(\mathbb{T})$, 
	$$\sum_{j\geq 0}\frac{|\widehat{f}(j)|}{j+1}\leq \pi \|f\|_{H^{1}(\mathbb{T})},$$ where the constant $\pi$ is the best possible and is due to J. Schur \cite{MR1580823}. To see this relation, we can notice that $j+1=\chi_{\mathbb{Z}_{+}}\ast \chi_{\mathbb{Z}_{+}}(j)$, and $\mathbb{Z}_{+}$ is the support of the Fourier transform of functions in $H^1(\mathbb{T})$. Thus, we are going to refer to such an inequality as Hardy's inequality for the set $\Omega$.
	\par In Section \ref{sec4} we study the validity of a more general inequality, which we will call adjusted Hardy's inequality, i.e. for which convex sets $\Omega$ and $d\in\mathbb{R}$, there exists $C(\Omega,d)>0$ such that 
	$$\int_{2P}\dfrac{|\widehat{f}(x)|}{\omega^{d}_{\Omega}(x)}dx\leq C(\Omega,d)\|f\|_{L^{1}},$$ for all $f\in \PW^1(2\Omega)$? We give some positive and some negative results which can be summarized in the following theorem.
	\begin{theo}\label{theo3}
		Let $\Omega$ be a convex subset of $\mathbb{R}^{n}$. Then the following hold.
		\begin{enumerate}
			\item If $\Omega$ is a polytope, then the adjusted Hardy's inequality holds if and only if $d\leq 1$.
			\item If $\Omega$ is bounded, then the adjusted Hardy's inequality holds for $d<\frac{2}{n+1}$.
			\item If $\Omega$ is free of lines, then the adjusted Hardy's inequality fails for all $d>1$. If furthermore $\Omega$ is unbounded, then the adjusted Hardy's inequality fails for all $d\neq 1$.
		\end{enumerate}
	\end{theo}
	\begin{Organisation}
		The paper is organized in 3 further sections. Sections \ref{sec2} and \ref{sec3} are devoted to the proofs of Theorem \ref{ball} and \ref{poly} respectively. In Section \ref{sec4} we discuss and prove some results on the adjusted Hardy's inequality.
	\end{Organisation}
	\begin{notat}
		We will use the notation $f\lesssim g$ whenever there exists a positive constant $C$ (possibly depending on parameters understood from context) such that $f\leq Cg$. We will also use the notation $f\approx g$ whenever $f\lesssim g$ and $g\lesssim f$.
	\end{notat}
	\begin{ack}
		The author would like to thank his supervisor Karl--Mikael Perfekt and Adri\'an Llinares for valuable help and comments. He would also like to thank Ole Fredrik Brevig and Romanos Diogenes Malikiosis, whose comments helped in the improvement of Theorem \ref{poly}.
	\end{ack}
	\section{Definitions and Proof of Theorem \ref{ball}}\label{sec2}
	Let $\Omega$ be an open convex subset of $\mathbb{R}^{n}$ and $\widehat{\phi}$ be a distribution in $2\Omega$. The Hankel operator $\Ha_{\phi}$ is densely defined by the formula 
	$$\widehat{\Ha_{\phi}f}(x)=\int_{\mathbb{R}^n}\widehat{\phi}(x+y)\widehat{f}(y)\chi_{\Omega}(x)\chi_{\Omega}(y)dy,$$ which can be considered as an operator on $\PW(\Omega)$ or $L^2(\mathbb{R}^n)$. Even though the distribution $\widehat{\phi}$ may not arise from a tempered distribution, we allow ourselves to refer to $\phi$. Proposition 5.1 of \cite{MR4227573} states that the space of bounded Hankel operators is a closed subspace of the bounded operators in $\PW(\Omega)$. Thus, Schatten class Hankel operators are closed subspaces of the Schatten class operators $S^{p}$, hence we have the following lemma. 
	\begin{lm}\label{implNehari}
		Suppose that for a convex set $\Omega\subset\mathbb{R}^{n}$, Nehari's theorem holds for some $p\geq 1$. Then there exists $C=C(\Omega,p)>0$ such that 
		$$\inf\{\|\psi\|_{L^{\infty}}:\widehat{\psi}|_{2\Omega}=\widehat{\phi}|_{2\Omega}\}\leq C\|\Ha_{\phi}\|_{S^{p}}.$$
	\end{lm}
	\begin{proof}
		Let $\mathcal{X}$ be the quotient Banach space $L^{\infty}(\mathbb{R}^{n})/\sim$, where $\phi\sim\psi$ whenever $\widehat{\phi}|_{2\Omega}=\widehat{\psi}|_{2\Omega}$, and let $\mathcal{X}_{p}=\{\phi\in \mathcal{X}:\Ha_{\phi}\in S^{p}(\PW(\Omega))\}$. Let also $\mathcal{H}$ be the bounded Hankel operators on $\PW(\Omega)$ and $\mathcal{H}_{p}=\mathcal{H}\cap S^{p}(\PW(\Omega))$. Then the operator $T:\mathcal{X}_{p}\to\mathcal{H}_{p}$, $T\phi=\Ha_{\phi}$ is a bijection by our assumption. Let now $\phi_{j}\in \mathcal{X}_{p}$, such that $\phi_{j}\to 0$ in $\mathcal{X}_{p}$ and $\Ha_{\phi_{j}}\to S$ in $S^{p}(\PW(\Omega))$, where $S$ is a Schatten class operator, $S\in S^{p}(\PW(\Omega))$. Then there exist $\psi_{j}\in L^{\infty}(\mathbb{R}^{n})$, $\phi_{j}\sim\psi_{j}$ such that $\|\psi_{j}\|_{L^{\infty}}\to 0$. This implies that
		$$\|S\|\leq \|\Ha_{\psi_{j}}-S\|+\|\Ha_{\psi_{j}}\|\leq \|\Ha_{\phi_{j}}-S\|_{S^{p}}+\|\psi_{j}\|_{L^{\infty}}\to 0.$$
		Therefore $S=0$ and by the closed graph theorem $T$ is bounded. Finally, by the open mapping theorem $T^{-1}$ is bounded as desired.
	\end{proof}
	Using H\"{o}lder's inequality we can see that for every $f$ Schwarz function with $\supp\widehat{f}\subset 2\Omega$, 
	\begin{eqnarray}\label{1}
		\dfrac{|\langle f,\phi\rangle|}{\|f\|_{L^{1}}}\leq \inf\{\|\psi\|_{L^{\infty}}:\widehat{\psi}|_{2\Omega}=\widehat{\phi}|_{2\Omega}\}.
	\end{eqnarray}
	Combining inequality \eqref{1} with Lemma \ref{implNehari}, to disprove Nehari's theorem for some $p\geq 1$ it suffices to find a sequence of Schwarz functions $f_{N}$, $\supp\widehat{f}_{N}\subset 2\Omega$, and a sequence of Schwarz functions $\phi_{N}$ such that 
	\begin{eqnarray}\label{2}
		\dfrac{|\langle f_{N},\phi_{N}\rangle|}{\|f_{N}\|_{L^{1}}\|\Ha_{\phi_{N}}\|_{S^{p}}}\to \infty.
	\end{eqnarray}
	The construction of these sequences are inspired by \cite{MR4502777}. Let $N$ be a fixed integer, and suppose that we can find Schwarz functions $\phi_{i}$, $i=1,...,N$ with Fourier transforms supported in $2\Omega$ such that the sets $D_{\phi_{i}}$ are disjoint, where 
	$$D_{\phi_{i}}=\Omega\cap(\supp\widehat{\phi_{i}}-\Omega).$$ 
	Let us set $\psi_{N}=\sum_{i=1}^{N}\phi_{i}$. The Hankel operators generated by these $\phi_{i}$ are orthogonal (see the proof of Lemma 2 in \cite{MR4700194}), hence \begin{eqnarray}\label{3}
		\|\Ha_{\psi_{N}}\|^{p}_{S^{p}}=\sum_{i=1}^{N}\|\Ha_{\phi_{i}}\|^{p}_{S^{p}}.
	\end{eqnarray}
	Since $\Ha_{\phi}^{\ast}=\Ha_{\psi}$ where $\widehat{\psi}=\overline{\widehat{\phi}}$, by \cite[Theorem 1]{MR500308}, we have that for $p>2$ the Schatten norm satisfies the bound $$\|\Ha_{\phi_{i}}\|_{S^{p}}\leq (\|\widehat{\phi}_{i}(x+y)\chi_{\Omega}(x)\chi_{\Omega}(y)\|_{p',p}\|\overline{\widehat{\phi}}_{i}(x+y)\chi_{\Omega}(x)\chi_{\Omega}(y)\|_{p',p})^{\frac{1}{2}},$$ 
	where for a function $g$ we use the notation
	$$\|g\|_{p',p}=\left(\int\left(\int |g(x,y)|^{p'}dx\right)^{\frac{p}{p'}}dy\right)^{\frac{1}{p}},$$ and $p'$ is the conjugate exponent of $p$, $\frac{1}{p}+\frac{1}{p'}=1$.
	Therefore using Minkowski's inequality we get that 
	\begin{eqnarray}\label{4}
		\|\Ha_{\phi_{i}}\|_{S^{p}}\leq \|\widehat{\phi_{i}}\omega_{\Omega}^{\frac{1}{p}}\|_{L^{p'}}.
	\end{eqnarray}
	Thus by \eqref {2}, \eqref{3} and \eqref{4}, to contradict Nehari's theorem for some $p\geq 1$ it suffices for every integer $N$ to find a sequence of Schwarz functions $\phi_{i}$ with disjoint $D_{\phi_{i}}$ and $\supp\widehat{\phi}_{i}\subset 2\Omega$, $i=1,...,N$, such that for $\psi_{N}=\sum_{i=1}^{N}\phi_{i}$,
	\begin{eqnarray}\label{5}
		\dfrac{\sum_{i=1}^{N}\|\phi_{i}\|^{2}_{L^{2}}}{\|\psi_{N}\|_{L^{1}}\sqrt[p]{\sum_{i=1}^{N}\|\widehat{\phi_{i}}\omega_{\Omega}^{\frac{1}{p}}\|^{p}_{L^{p'}}}}\to \infty. 
	\end{eqnarray}
	The following lemma makes use of the left hand side term in \eqref{5}, combined with some geometrical assumptions, to derive the existence of a uniform upper bound, assuming the Nehari theorem for some $p$.
	\begin{lm}\label{generalnegation}
		Let $\Omega$ be a convex subset of $\mathbb{R}^{n}$. Suppose that for $\epsilon>0$ we can find $N$ vectors $y_{1},...,y_{N}\in \partial\Omega$, $N$ vectors $x_{1},...,x_{N}\in\Omega$ and $r>0$ such that the following hold:
		\begin{enumerate}
			\item \label{prop1} $B(y_{i},\epsilon)\cap B(y_{j},\epsilon)=\emptyset$ for $i\neq j$.
			\item \label{prop2} $\overline{B}(x_{i},r)\subset\Omega$.
			\item \label{prop3} $(2\overline{B}(x_{i},r)-\Omega)\cap\Omega\subset B(y_{i},\epsilon).$
		\end{enumerate}
		Let us set $a=\max_{i=1,...,N}\sup_{x\in 2B(x_{i},r)}\omega_{\Omega}(x)$. If Nehari's theorem holds for some $p\geq 1$, then for every $d>0$ there exists $C(\Omega,p,d)>0$ such that, $$N^{\frac{d}{n+2d}-\frac{1}{p}}\left(\frac{r^{n}}{a}\right)^{\frac{1}{p}}\leq C(\Omega,p,d).$$
	\end{lm}
	\begin{proof}
		Let $\phi$ be a smooth function in $\mathbb{R}^{n}$ such that $0\leq \widehat{\phi}\leq 1$, $\widehat{\phi}=1$ on $\frac{1}{2}B(0,1)$ and zero outside $B(0,1)$. We define $\widehat{\phi}_{i}(x)=\widehat{\phi}(\frac{x-2x_{i}}{2r})$ and $\psi=\sum_{i=1}^{N}\phi_{i}$. Then, if Nehari's theorem holds for certain $p\geq 1$, by \eqref{1}, \eqref{3}, \eqref{4} and Lemma \ref{implNehari} we have that there exists a constant $C(\Omega,p)>0$ such that
		$$ \dfrac{\sum_{i=1}^{N}\|\widehat{\phi}_{i}\|_{L^{2}}^{2}}{\|\psi\|_{L^{1}}\sqrt[p]{\sum_{i=1}^{N}\|\widehat{\phi}_{i}\omega_{\Omega}^{\frac{1}{p}}\|_{L^{p'}}^{p}}}\leq \dfrac{\|\widehat{\psi}\|^{2}_{L^{2}}}{\|\psi\|_{L^{1}}\|\Ha_{\psi}\|_{S^{p}}}\leq C(\Omega,p).$$
		Since by definition $\supp\widehat{\phi}_{i}=2B(x_{i},r)$, we get that $\|\widehat{\phi}_{i}\|_{L^{2}}\approx r^{\frac{n}{2}}$ and $\|\widehat{\phi}_{i}\omega_{\Omega}^{\frac{1}{p}}\|_{L^{p'}}\lesssim r^{\frac{n}{p'}}a^{\frac{1}{p}}$. Regarding the norm $\|\psi\|_{L^{1}}$ we proceed as follows. Let $R>0$ and set $$\|\psi\|_{L^{1}}=\int_{|x|\leq\frac{R}{2r}}|\psi|+\int_{|x|>\frac{R}{2r}}|\psi|:=I_{1}+I_{2}.$$
		We first bound $I_{1}$ using the Cauchy--Schwarz inequality,
		$$I_{1}\leq \|\psi\|_{L^{2}}\sqrt{m(B(0,R/2r))}\approx \left(\sum_{i=1}^{N}\|\phi_{i}\|_{L^{2}}^{2} \right)^{\frac{1}{2}} (R/r)^{\frac{n}{2}}\approx(R/r)^{\frac{n}{2}}\sqrt{Nr^{n}}=R^{\frac{n}{2}}\sqrt{N}.$$
		For $I_{2}$ we compute 
		$$I_{2}\leq \sum_{i=1}^{N}\int_{|x|>\frac{R}{2r}}|\phi_{i}|=\sum_{i=1}^{N}\int_{|x|>R}|\phi|=N\int_{|x|>R}\frac{||x|^{d}\phi|}{|x|^{d}}\lesssim \frac{N}{R^{d}},$$ where the last estimate holds with a constant depending on $d$. Combining these two inequalities we get that
		$$\|\psi\|_{L^{1}}\lesssim R^{\frac{n}{2}}\sqrt{N}+R^{-d}N.$$ Since we want to minimize this upper bound, we set $R=N^{\frac{1}{n+2d}}$ and get the bound $\|\psi\|_{L^{1}}\lesssim N^{\frac{n+d}{n+2d}}.$ Therefore we get that there exists $C(\Omega,p,d)>0$ such that 
		$$\dfrac{Nr^{n}}{N^{\frac{n+d}{n+2d}}\sqrt[p]{\sum_{i=1}^{N}ar^{n(p-1)}}}\leq C(\Omega,p,d),$$ which simplifies to $$N^{\frac{d}{n+2d}-\frac{1}{p}}\left(\frac{r^{n}}{a}\right)^{\frac{1}{p}}\leq C(\Omega,p,d),$$ as desired.
	\end{proof}
	In order to control the number $a$ of Lemma \ref{generalnegation}, we need to be able to compute the $\omega$ function of the $n$--dimensional ball, thus we need the following lemma.
	\begin{lm}\label{computew}
		It is true that 
		$$\omega_{B(0,1)}(x)\approx (2-|x|)^{\frac{n+1}{2}}\chi_{2B(0,1)}(x).$$ 
	\end{lm}
	\begin{proof}
		Since both quantities are $0$ outside $2B(0,1)$, we only need to treat the case $x\in 2B(0,1)$. Also, by rotational symmetry it suffices to prove it for $x\in(0,2)\times\{0\}^{n-1}$. Furthermore, by the symmetry of the ball, we can see that the $\omega_{B(0,1)}(x)$ equals twice the measure of $B(0,1)\cap \{(t_1,...,t_n)\in\mathbb{R}^{n}:t_{1}>\frac{|x|}{2}\}$. Now by Cavalieri's principle we can compute 
		$$\omega_{B(0,1)}(x)=2\int_{|x|/2}^{1}m_{n-1}(B(0,1)\cap \{t_{1}=t\})dt,$$ where $m_{n-1}$ is the $n-1$ dimensional Lebesgue measure. We can see that the sets $B(0,1)\cap \{t_{1}=t\}$ are $n-1$ dimensional balls of radii $\sqrt{1-t^{2}}$, therefore we get that 
		$$\omega_{B(0,1)}(x)\approx \int_{|x|/2}^{1}(1-t^{2})^{\frac{n-1}{2}}dt\approx \int_{|x|/2}^{1}(1-t)^{\frac{n-1}{2}}dt\approx (2-|x|)^{\frac{n+1}{2}},$$ as desired.
	\end{proof}
	The following lemma will help our understanding on the relation between the $x_{i}$'s and $r$ in Lemma \ref{generalnegation} for the case of the ball, which will later help us approach the more general case. 
	\begin{lm}\label{disc}
		There exist $C,\epsilon_{0}>0$ such that for every $\epsilon\in(0,\epsilon_{0})$ it is true that 
		$$(2\overline{B}((1-C\epsilon^{2})x,C\epsilon^{2})-B(0,1))\cap B(0,1)\subset B(x,\epsilon),$$ for all $x\in \partial B(0,1).$
	\end{lm}
	\begin{proof}
		By rotational symmetry, it suffices to prove it for $x=(1,0,...,0)$. By straightforward computation (see Figure \ref{fig:1a}) we notice that for big enough $s\in (0,2)$, say $s\in (s_{0},2)$,
		$$\big(sx-B(0,1)\big)\cap B(0,1)\subset B(x,\sqrt{2-s}).$$ 
		Therefore we get that for every $u\in\partial B(0,1)$,
		$$\big(su-B(0,1)\big)\cap B(0,1)\subset B(u,\sqrt{2-s}).$$
		Thus for $s\in (\frac{s_{0}+2}{2},2)$ we get that
		\begin{multline*}
			\big(\overline{B}(sx,2-s)-B(0,1))\cap B(0,1)\subset \bigcup_{ru\in \overline{B}(sx,2-s)} B(u,\sqrt{2-r}) \\
			\subset \bigcup_{ru\in \overline{B}(sx,2-s)} B(x,\sqrt{2-r}+|u-x|),
		\end{multline*}
		where, in the product $ru$, $r$ is positive and $u$ belongs to the unit circle $\partial B(0,1).$ 
		Since $ru\in \overline{B}(sx,2-s)$, we have that $r=|sx+(2-s)v|$, $v\in \overline{B}(0,1)$, thus
		$$2-r=2-|sx+(2-s)v|\leq 2-\big|s-|2-s||u|\big|\leq 4-2s. $$
		and as can be seen by Figure \ref{fig:1b}, $|x-u|\lesssim 2-s$. Therefore $\sqrt{2-r}+|u-x|\lesssim\sqrt{2-s}+(2-s)\lesssim \sqrt{2-s}$. This completes the proof.
	\end{proof}
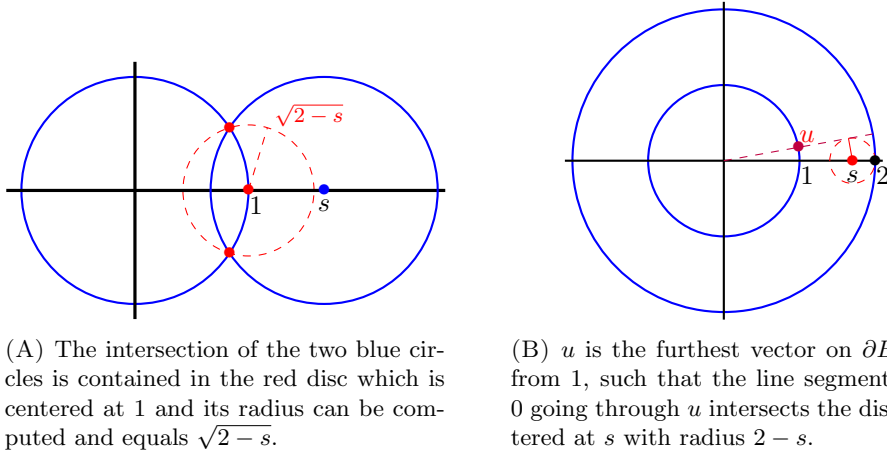
\begin{figure}
	\centering
	\begin{subfigure}[t]{0.46\textwidth}
		\centering
		\begin{tikzpicture}
			\draw[red, dashed] (1.5,0) circle (0.866cm);
			\draw[red, dashed](1.5,0) -- (1.75,0.8291);
			\node[red] at (2.3,1) {\footnotesize $\sqrt{2-s}$};
			\draw[blue, thick] (0,0) circle (1.5cm);
			\draw[very thick] (-1.7,0) -- (4.1,0);
			\draw[blue, thick] (2.5,0) circle (1.5cm);
			\draw[very thick] (0,-1.7) -- (0,1.7);
			\node[red] at (1.5,0) {$\bullet$};
			\node at (1.6,-0.2) {1};
			\node[blue] at (2.5,0) {$\bullet$};
			\node at (2.5,-0.2) {$s$};
			\node[red] at (1.25,0.8291) {$\bullet$};
			\node[red] at (1.25,-0.8291) {$\bullet$};
		\end{tikzpicture}
		\caption{The intersection of the two blue circles is contained in the red disc which is centered at $1$ and its radius can be computed and equals $\sqrt{2-s}$.}
		\label{fig:1a}
	\end{subfigure}
	\hspace{0.05\textwidth}
	\begin{subfigure}[t]{0.46\textwidth}
		\centering
		\begin{tikzpicture}
			\draw[blue, thick] (0,0) circle (1cm);
			\draw[blue, thick] (0,0) circle (2cm);
			\draw[thick] (-2.1,0) -- (2.1,0);
			\draw[thick] (0,-2.1) -- (0,2.1);
			\draw[red, dashed] (1.7,0) circle (0.3cm);
			\draw[purple, dashed] (0,0) -- (1.968611,0.35294);
			\node[purple] at (0.9843055,0.17647) {$\bullet$};
			\node[red] at (1.1,0.32) {$u$};
			\node at (1.7,-0.2) {$s$};
			\node[red] at (1.7,0) {$\bullet$};
			\node at (2,0) {$\bullet$};
			\node at (1.1,-0.2) {1};
			\node at (2.1,-0.2) {2};
			\draw[red] (1.7,0) -- (1.6470,0.2952);
		\end{tikzpicture}
		\caption{$u$ is the furthest vector on $\partial B(0,1)$ from $1$, such that the line segment from $0$ going through $u$ intersects the disc centered at $s$ with radius $2-s$.}
		\label{fig:1b}
	\end{subfigure}
	\caption{Some useful geometric observations.}
	\label{fig:main}
\end{figure}
	\par Let us denote with $T(\alpha,\beta)$, $\alpha,\beta>0$, the pyramid $$T(\alpha,\beta)=\{(x_{1},...,x_{n-1},x_{n})\in\mathbb{R}^{n}:|x_{1}|,...,|x_{n-1}|<\alpha-\frac{\alpha}{\beta}x_{n},0<x_{n}<\beta\}.$$
	In two dimensions $T(\alpha,\beta)$ is an isosceles triangle of base the line segment $(-\alpha,\alpha)\times\{0\}$ and vertex $(0,\beta)$. In three dimensions $T(\alpha,\beta)$ is a square based pyramid with base the square $(-\alpha,\alpha)^{2}\times\{0\}$ and apex $(0,0,\beta)$.
	By a simple trigonometric argument we can observe that for $t\in (\frac{\beta}{2},\beta)$, it is true that 
	\begin{eqnarray}\label{6}
		B((0,...,0,t),\frac{\alpha}{\sqrt{\alpha^{2}+\beta^{2}}}(\beta-t))\subset T(\alpha,\beta)
	\end{eqnarray} 
	We now define the following family of boundary vectors of a convex set $\Omega$. For $\alpha,\beta>0$ and $R>0$ we define $\mathscr{R}(\alpha,\beta,R)$ to be the set of all boundary vectors $y\in\partial\Omega$ such that the following hold.
	\begin{enumerate}
		\item There exists $x\in\mathbb{R}^{n}$ such that $\Omega\subset B(x,R)$ and $|x-y|=R$.
		\item There exists an isometry $I:\mathbb{R}^{n}\to\mathbb{R}^{n}$ such that, $$I(T(\alpha,\beta))\subset \Omega, \quad I(0,0,...,0,\beta)=y,$$ and $I$ maps the line segment $\{(0,0,...,0,t):t\in (0,\beta)\}$ into the line that is perpendicular to the supporting hyperplane of $B(x,R)$ at $y$.
	\end{enumerate}
	The existence of such a pyramid guarantees that we can find balls of suitable radii within 
	$\Omega$, and it also aids in visualizing how these balls can exist within the domain. The specific choice of the pyramid is not essential—we could replace it with any other polygon—but we opt for the pyramid for computational convenience.
	\begin{example*}
		In Figure \ref{fig:2} we see that if $\Omega$ is the ellipse $\frac{x^2}{\gamma^2}+\frac{y^2}{\beta^2}<1$, then the vector $(0,\beta)$ belongs to the class $\mathscr{R}(\alpha,\beta,\beta)$, $\alpha\leq \gamma$. The red triangle is the set $T(\alpha,\beta)$, and the radius of the green circle is $\frac{\alpha}{\sqrt{\alpha^2+\beta^2}}(\beta-t)$.
	\end{example*}
 We now give a lemma that verifies properties \eqref{prop1}--\eqref{prop3} for each $y\in\partial\Omega$ individually that belongs to some set $\mathscr{R}(\alpha,\beta,R)$, by making use of Lemma \ref{disc} and our observation for $T(\alpha,\beta)$.
	\begin{figure}
		\centering
		\begin{tikzpicture}[scale=3] 
			\draw[blue, thick] (0,0) circle (1);
			
			\draw[black] (-1.1,0) -- (1.1,0);
			\draw[black] (0,-1.1) -- (0,1.1);
			
			
			\draw[red, thick] (0,1) -- (0.5,0) -- (-0.5,0) -- cycle;
			\draw[thick] (0,0.6) -- (0.163,0.674);
			\draw[green, thick] (0,0.6) circle (0.179);
			
			\filldraw[black] (0,1) circle (1pt);
			\node at (-0.04,1.1) {$\beta$};
			\filldraw[black] (0.5,0) circle (1pt) node[below] {$\alpha$};
			\filldraw[black] (-0.5,0) circle (1pt) node[below] {$-\alpha$};
			\node at (-0.04,0.6) {$t$};
			\draw[thick] (0,0) ellipse (0.76 and 1);
			\filldraw[black] (0,0) circle (1pt) node[below left] {$0$};
			\filldraw[black] (-0.76,0) circle (1pt) node[below left] {$-\gamma$};
			\filldraw[black] (0.76,0) circle (1pt) node[below right] {$\gamma$};
		\end{tikzpicture}   
		\caption{If $\Omega$ is the ellipse $\frac{x^2}{\gamma^2}+\frac{y^2}{\beta^2}<1$ then the vector $(0,\beta)$ belongs to the class $\mathscr{R}(\alpha,\beta,\beta)$.}
		\label{fig:2}
	\end{figure}
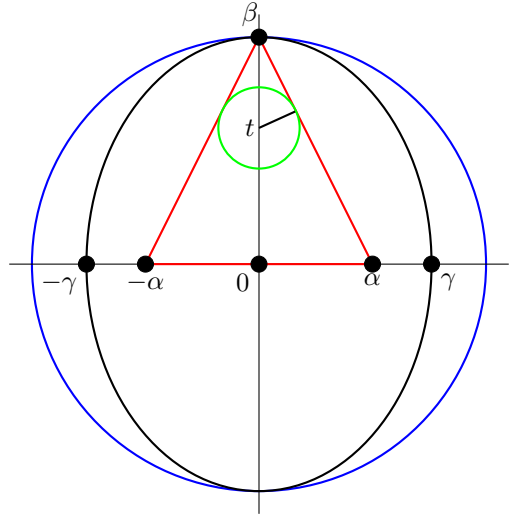
	\begin{lm}\label{indiv}
		Let $\Omega$ be a bounded convex subset of $\mathbb{R}^{n}$. Let us fix $\alpha,\beta,R>0$. Then there exist positive constants $C_1,C_2$ and $k$ such that for every $y\in\mathscr{R}(\alpha,\beta,R)$ and every $\epsilon<k$ we can find $x_{\epsilon}\in\Omega$ such that 
		\begin{enumerate}
			\item $B(x_{\epsilon},C_1\epsilon^2)\subset \Omega$.
			\item $(2\overline{B}(x_{\epsilon},C_1\epsilon^2)-\Omega)\cap\Omega\subset B(y,\epsilon)$.
			\item $\omega_{\Omega}(x)\leq C_2 \epsilon^{n+1}$, for all $x\in 2B(x_{\epsilon},C_1\epsilon^2)$.
		\end{enumerate}
	\end{lm}
	\begin{proof}
		Let, without loss of generality $x=0$, $R=1$ and $y=(1,0,...,0)$, where $x$ is the center of the ball that is introduced in the definition of $\mathscr{R}(\alpha,\beta,R)$. Let $C,\epsilon_{0}$ be as in Lemma \ref{disc} and set $x_{\epsilon}=(1-C\epsilon^{2})y$ and $t_{\epsilon}=C\epsilon^{2}$. First, we notice that $|x_{\epsilon}-y|= C\epsilon^{2}$.
		By our observation \eqref{6} for $T(\alpha,\beta)$, for each $x_{\epsilon}$ with $|x_{\epsilon}-y|<\frac{\beta}{2}$, or in other words for $ \epsilon<\sqrt{\frac{\beta}{2C}}$ we have that 
		$$B(x_{\epsilon},K(\alpha,\beta)\epsilon^{2})\subset I(T(\alpha,\beta))\subset \Omega,$$ where $K(\alpha,\beta)$ is a suitable constant.  By Lemma \ref{disc} we also have that
		$$(2\overline{B}(x_{\epsilon},t_{\epsilon})-\Omega)\cap\Omega\subset (2\overline{B}(x_{\epsilon},t_{\epsilon})-B(0,1))\cap B(0,1)\subset B(y,\epsilon).$$ Thus for $C_1=\min(C,K(\alpha,\beta))$ we have proven $(1)$ and $(2)$. Finally, to verify $(3)$, since $\omega_{\Omega}\leq \omega_{B(0,1)}$, it suffices to prove it for $\Omega=B(0,1)$. Let $ru\in B(x_{\epsilon},C_1\epsilon^2)$, $r$ positive and $u$ on the unit circle $\partial B(0,1)$.
		Then by Lemma \ref{computew} it suffices to bound $1-r$ by a constant multiple of $\epsilon^{2}$. This follows by the fact that $r=|x_{\epsilon}+C_1\epsilon^2v|$ for some $v\in B(0,1)$. 
	\end{proof}
	Therefore, we are now able to apply Lemma \ref{generalnegation} to sets with multiple boundary vectors that satisfy Lemma \ref{indiv}.
	\begin{cor}\label{cor 1}
		Let $\Omega$ be a convex bounded subset of $\mathbb{R}^{n}$. Let $\alpha,\beta,R>0$ and suppose that for $\epsilon>0$ we can find $N$ vectors $y_{1},...,y_{N}\in\mathscr{R}(\alpha,\beta,R)$ of mutual distance greater than $2\epsilon$. Then if Nehari's theorem holds for some $p\geq 1$, there exists a constant $C=C(\Omega,p,d,\alpha,\beta,R)$ such that 
		$$N^{\frac{d}{n+2d}-\frac{1}{p}}\epsilon^{\frac{n-1}{p}}\leq C.$$
	\end{cor}
	\begin{proof}
		By Lemma \ref{indiv}, we can find $x_{1},...,x_{N}\in \Omega$ and $s\approx \epsilon^{2}$ such that 
		\begin{enumerate}
			\item $B(x_{i},s)\subset \Omega$.
			\item $(2\overline{B}(x_{i},s)-\Omega)\cap\Omega\subset B(y_{i},\epsilon)$.
			\item $\omega_{\Omega}(x)\lesssim \epsilon^{n+1}$, for all $x\in 2B(x_{i},s)$,
		\end{enumerate}
		for $i=1,...,N$. Therefore by Lemma \ref{generalnegation}, we can take $a\approx\epsilon^{n+1}$ and $r\approx \epsilon^{2}$, thus there exists a constant $C=C(\Omega,p,d,\alpha,\beta,R)$ such that 
		$$ N^{\frac{d}{n+2d}-\frac{1}{p}}\epsilon^{\frac{n-1}{p}}\lesssim N^{\frac{d}{n+2d}-\frac{1}{p}}(\dfrac{r^{n}}{a})^{\frac{1}{p}}\leq C,$$ as desired.
	\end{proof}
   Controlling $N$ from bellow by a power of $\epsilon$ and then allowing $\epsilon$ to approach $0$ in Corollary \ref{cor 1}, we are able to negate Nehari's theorem for sufficiently large $p$.
	\begin{cor}\label{cor 2}
		Let $\Omega$ be a convex bounded subset of $\mathbb{R}^{n}$. Suppose that for some $\alpha,\beta,R>0$ there exist $k,C>0$ such that for every $\epsilon\in (0,1)$ we can find $N\geq C \epsilon^{-k}$ vectors $y_{1},...,y_{N}\in\mathscr{R}(\alpha,\beta,R)$ of mutual distance greater than $2\epsilon$. Then Nehari's theorem fails for all $p>\frac{2}{k}(n+k-1)$.
	\end{cor}
	\begin{proof}
		By Corollary \ref{cor 1} we get that $$\epsilon^{-\frac{d k}{n+2d}+\frac{k}{p}}\epsilon^{\frac{n-1}{p}}\leq C(\Omega,p,d,\alpha,\beta,R).$$ By taking $\epsilon\to 0$, the left part of the inequality diverges whenever $-\frac{d k}{n+2d}+\frac{k}{p}+\frac{n-1}{p}<0$ or equivalently $p>\frac{n+2d}{d k}(n+k-1)$. Since this holds for all $d >0$, we get the result by taking $d\to \infty$.
	\end{proof}
	Using Corollary \ref{cor 2} we can now prove Theorem \ref{ball}. \\
	\\ \medspace \textbf{Proof of Theorem \ref{ball}.} We will prove the result only for bounded sets, since the argument is local. Let $K\subset \partial\Omega$ be open and $C^2$. Since the curvature is non-zero, we can assume that the curvature at $K$ is bounded from below, and thus it is evident that we can find uniform $\alpha,\beta$ and $R$, such that every point in $K$ also belongs to $\mathscr{R}(\alpha,\beta,R)$. Therefore, by Corollary \ref{cor 2}, it suffices, for every $\epsilon\in (0,1)$ to find $N \gtrsim \epsilon^{1-n}$ vectors in $K$ of distance greater than $\epsilon.$ This is trivial since $K$ is a $n-1$ dimensional $C^2$ surface.\qed \\
	\par This argument cannot be directly applied for every bounded convex set that is not a polygon. Consider for example a set that is the infinite union of line segments on the plane, such that the angles become bigger, approaching $\pi$. In this case we are not able to find a uniform $R>0$, thus a more detailed calculation is required that also contains the behavior of $R$ with respect to $N$ and $\epsilon$, which will give weaker and case-specific results.
	
	\section{Hardy's inequality}\label{sec3}
	Let $\Omega$ be a convex subset of $\mathbb{R}^{n}$, and from now on when we refer to a convex set we will also assume that it does not contain lines. Then, $\omega_{\Omega}(x)$ exists for all $x\in\mathbb{R}^{n}$. For convenience of notation, we understand the function $\omega_{\Omega}^{-1}$ to be zero outside $2\Omega$. We are interested in the following question; for which convex sets $\Omega$ there exists a constant $C>0$ such that the inequality \begin{eqnarray} \label{7}
		\int_{\mathbb{R}^{n}}\dfrac{|\widehat{f}(x)|}{\omega_{\Omega}(x)}dx\leq C\|f\|_{L^{1}},
	\end{eqnarray}
	holds for all $f\in\PW^{1}(2\Omega)$? If such an inequality holds we will refer to it as Hardy's inequality for $\Omega$. 
	\par For a convex set $\Omega$, Hankel operators on the Paley--Wiener space $\PW(\Omega)$ are integral operators on $L^2(\mathbb{R}^n)$ with kernel $\widehat{\phi}(x+y)\chi_{\Omega}(x)\chi_{\Omega}(y)$, $\phi$ being the symbol of the operator, and thus we can compute its Hilbert--Schmidt norm to be 
	$$\|\Ha_\phi\|_{S^2}=\|\widehat{\phi}\sqrt{\omega_{\Omega}}\|_{L^2}.$$
	Using this fact we start by proving that inequality \eqref{7} implies Nehari's theorem for $p=2$. 
	\begin{lm}\label{neh2}
		Let $\Omega$ be a convex and bounded subset of $\mathbb{R}^n$. Then the inequality $$\int_{\mathbb{R}^{n}}\dfrac{|\widehat{f}(x)|^2}{\omega_{\Omega}(x)}dx\leq C\|f\|_{L^{1}}^2,$$ for all $f\in\PW^1(2\Omega)$ and some $C=C(\Omega)>0$ is equivalent to Nehari's theorem for $\Omega$ for $p=2$. Consequently, Hardy's inequality implies Nehari's theorem for $\Omega$ for $p=2$.
	\end{lm}
	\begin{proof}
		It is evident that Hardy's inequality implies the inequality \begin{equation}\label{Helson} \int_{\mathbb{R}^{n}}\dfrac{|\widehat{f}(x)|^2}{\omega_{\Omega}(x)}dx\leq C\|f\|_{L^{1}}^2,
			\end{equation} and thus it suffices to prove the equivalence.
		Suppose that inequality \eqref{Helson} holds for $\Omega$. Then, for two functions $f,\phi$, with $\supp\widehat{f}\subset 2\Omega$, by the Cauchy--Schwarz inequality and \eqref{Helson} we have that \begin{eqnarray}
			|\langle \phi,f\rangle|&=&\big|\int_{\mathbb{R}^{n}}\widehat{\phi}(x)\overline{\widehat{f}}(x)dx\big|=\big|\int_{\mathbb{R}^{n}}\widehat{\phi}(x)\sqrt{\omega_{\Omega}(x)}\dfrac{\overline{\widehat{f}}(x)}{\sqrt{\omega_{\Omega}(x)}}dx\big| \nonumber \\
			&\leq&\|\widehat{\phi}\sqrt{\omega_{\Omega}}\|_{L^2} \|\dfrac{\widehat{f}}{\sqrt{\omega_{\Omega}}}\|_{L^2} \leq C\|f\|_{L^{1}}\|\Ha_{\phi}\|_{S^{2}}. \nonumber
		\end{eqnarray}
		Thus, if $\Ha_\phi \in S^2(\PW(\Omega))$, the operator $T_\phi(f)=\langle f,\phi\rangle$ is a bounded functional in $\PW^{1}(2\Omega)$, and by the Hahn--Banach theorem, $T_\phi$ extends continuously on $L^{1}(\mathbb{R}^{n})$, thus there exists a bounded function $\psi\in L^{\infty}(\mathbb{R}^{n})$ such that $T_\phi(f)=\langle f,\psi\rangle$ for all $f\in L^{1}(\mathbb{R}^n)$, hence $\Ha_\phi=\Ha_\psi$. 
		\par Now for the converse, suppose that Nehari's theorem holds for $\Omega$ for $p=2$. Let us fix $z\in \Omega$ and for $r\in (0,1)$ we set $\Omega_r=r\Omega+(1-r)z$. For $f\in \PW^1(2\Omega_r)$, we set $g$ such that $\widehat{g}=\dfrac{\overline{\widehat{f}}}{\omega_{\Omega}}$. Then, Lemma \ref{implNehari} gives that,
		$$\int_{\mathbb{R}^n}\dfrac{|\widehat{f}(x)|^2}{\omega_{\Omega}(x)}dx=\langle f, g \rangle \leq C\|f\|_{L^{1}} \|\Ha_g\|_{S^{2}}= C\|f\|_{L^{1}} \|\widehat{g}\omega_{\Omega}^{\frac{1}{2}}\|_{L^{2}}=C\|f\|_{L^{1}} \|\widehat{f}\omega_{\Omega}^{-\frac{1}{2}}\|_{L^{2}}.$$
		Dividing both sides by $\|\widehat{f}\omega_{\Omega}^{-\frac{1}{2}}\|_{L^{2}}$ we get that
		$$ \left(\int_{\mathbb{R}^n}\dfrac{|\widehat{f}(x)|^2}{\omega_{\Omega}(x)}dx\right)^{\frac{1}{2}} \leq C\|f\|_{L^{1}},$$ for all $f\in \PW^1(2\Omega_r)$. Note that for the above division, $\|\widehat{f}\omega_{\Omega}^{-\frac{1}{2}}\|_{L^{2}}<\infty$ since $f\in \PW^1(2\Omega_r)$. 
		 To extend this to all $f\in \PW^1(2\Omega)$, we apply the above inequality to $f_r$, $\widehat{f}_r(x)=\widehat{f}(\frac{x-2(1-r)z}{r})$. Since $\|f_r\|_{L^1}=\|f\|_{L^1}$ and $f_r\in\PW^1(2\Omega_r)$, by Fatou's lemma we get that for $f\in\PW^1(2\Omega)$,
		$$ \int_{\mathbb{R}^n}\dfrac{|\widehat{f}(x)|^2}{\omega_{\Omega}(x)}dx\leq \liminf_{r\to 1}\int_{\mathbb{R}^n}\dfrac{|\widehat{f}_r(x)|^2}{\omega_{\Omega}(x)}dx\leq \liminf_{r\to 1}C\|f_r\|^2_{L^1}=C\|f\|^2_{L^1},$$
		as desired.
	\end{proof}
    For the case of the real half-line $\mathbb{R}_{+}$, it is known that Hardy's inequality holds, and is a consequence of Nehari's theorem. For the sake of completeness we provide a proof.
    \begin{lm}\label{Hardy}
    	Hardy's inequality holds for $\mathbb{R}_{+}$ with constant $\pi$, i.e. $$\int_{0}^{\infty}\dfrac{|\widehat{f}(x)|}{x}dx\leq \pi\|f\|_{L^{1}},$$
    	for all $f\in\PW^{1}(\mathbb{R}_{+})$.
    \end{lm}
\begin{proof}
	Since $\PW^1(\mathbb{R}_{+})$ can be identified with the Hardy space on the upper half plane $H^1(\mathbb{R})$, by the Riesz factorization theorem every $f\in\PW^1(\mathbb{R}_{+})$ can be written as $f=gh$, $g,h\in\PW(\mathbb{R}_{+})$ and $\|f\|_{L^1}=\|g\|_{L^2}\|h\|_{L^2}$. Let $\psi(x)=i\pi\sgn(x)$. It is known that $\widehat{\psi}=$p.v.$\frac{1}{x}$, and $\widehat{\psi}=\frac{1}{x}$ in $\mathbb{R}_{+}$ in the distributional sense. Let us set $\widehat{G}=|\widehat{g}|$ and $\widehat{H}=|\widehat{h}|$. Then, we have that
	\begin{eqnarray}
		\int_{0}^{\infty}\dfrac{|\widehat{f}(x)|}{x}dx&\leq& 	\int_{0}^{\infty}\dfrac{\widehat{G}\ast \widehat{H}(x)}{x}dx=\langle GH,\psi\rangle \leq \|G\|_{L^2}\|H\|_{L^2}\|\psi\|_{L^{\infty}} \nonumber \\
		&=& \pi\|g\|_{L^2}\|h\|_{L^2}=\pi\|f\|_{L^1}. \nonumber 
	\end{eqnarray} 
The proof is complete.
\end{proof}
 The next lemma demonstrates a technique that we will use again in Lemma \ref{decomp} and shows that Hardy's inequality holds for intervals.
	\begin{lm}\label{interval}
		Let $I$ denote the interval $(-1/2,1/2)$. Then Hardy's inequality holds for $I$, i.e. there exists $C>0$ such that
		$$\int_{-1}^{1}\dfrac{|\widehat{f}(x)|}{1-|x|}dx\leq C\|f\|_{L^1},$$ for all $f\in \PW^1(2I)$.
	\end{lm}
	\begin{proof}
		Let $\phi_{-1},\phi_{1}$ two Schwarz functions such that $\supp\widehat{\phi}_{-1}\subset (-2,\frac{1}{2})$, $\supp\widehat{\phi}_{1}\subset (-\frac{1}{2},2)$ and $|\widehat{\phi}_{-1}(x)|+|\widehat{\phi}_{1}(x)|=1$ in $2I$. Then we can see that, for $f\in \PW^{1}(2I)$, 
		\begin{eqnarray}
			\int_{-1}^{1}\frac{|\widehat{f}(x)|}{1-|x|}dx&=& \int_{-1}^{1}\frac{|\widehat{f}(x)\widehat{\phi}_{1}(x)|}{1-|x|}dx+\int_{-1}^{1}\frac{|\widehat{f}(x)\widehat{\phi}_{-1}(x)|}{1-|x|}dx \nonumber \\
			&\lesssim& \int_{-\infty}^{1}\frac{|\widehat{f}(x)\widehat{\phi}_{1}(x)|}{1-x}dx+\int_{-1}^{\infty}\frac{|\widehat{f}(x)\widehat{\phi}_{-1}(x)|}{1+x}dx \nonumber \\
			&\lesssim& \|f\ast\phi_{1}\|_{L^{1}}+\|f\ast\phi_{-1}\|_{L^{1}}\lesssim \|f\|_{L^{1}}, \nonumber
		\end{eqnarray} 
		where the first inequality holds since $\supp\widehat{f}\widehat{\phi}_{1}\subset (-\frac{1}{2},1)$, $\supp\widehat{f}\widehat{\phi}_{-1}\subset (-1,\frac{1}{2})$ and the second inequality follow by Lemma \ref{Hardy}. 
	\end{proof}
	In order to lift this property to more dimensions, we will replicate the proof of H. Helson for the case of Hardy spaces on the multi--discs $\mathbb{T}^{n}$, \cite{MR2263964}, which was inspired by a technique used by A. Bonami \cite{MR283496}.
	\begin{lm}\label{prod}
		Let $\Omega_{1}\subset\mathbb{R}^{n},\Omega_{2}\subset\mathbb{R}^{m}$ such that Hardy's inequality holds. Then Hardy's inequality holds for $\Omega_{1}\times\Omega_{2}\subset \mathbb{R}^{n+m}$. Furthermore, $C(\Omega_1\times \Omega_2)=C(\Omega_1)C(\Omega_2)$, where by $C(A)$ we denote the best constant for Hardy's inequality for the set $A$.
	\end{lm}
	\begin{proof}
		First, by definition we can notice that for $x\in \mathbb{R}^{n}$, $y\in\mathbb{R}^{m}$, $\omega_{\Omega_{1}\times\Omega_{2}}(x,y)=\omega_{\Omega_{1}}(x)\omega_{\Omega_{2}}(y)$. Let $F_n$ be the Fourier transform of a function in $L^{1}(\mathbb{R}^{n+m})$ on the first $n$ variables, and $F_m$ on the other $m$ variables, that is 
		$$F_n f(x,y)=\int_{\mathbb{R}^n}f(\xi,y)e^{-2\pi i \langle \xi,x\rangle}d\xi\text{   and   } F_m f(x,y)=\int_{\mathbb{R}^m}f(x,\xi)e^{-2\pi i \langle \xi,y\rangle}d\xi.$$
		Then the classical Fourier in $\mathbb{R}^{n+m}$ is $\widehat{f}(x,y)=F_n F_m f(x,y)=F_m F_n f(x,y)$. Using Hardy's inequality for each variable and Fubini, we can see that for $f\in\PW^{1}(2(\Omega_{1}\times\Omega_{2}))$,
		\begin{eqnarray}
			\int_{\mathbb{R}^{n+m}}\dfrac{|\widehat{f}(x,y)|}{\omega_{\Omega_1\times\Omega_2}(x,y)}dxdy&=& \int_{\mathbb{R}^{m}}\dfrac{1}{\omega_{\Omega_2}(y)}\int_{\mathbb{R}^{n}}\dfrac{|F_nF_mf(x,y)|}{\omega_{\Omega_1}(x)}dxdy \nonumber \\ &\leq&  C(\Omega_1)\int_{\mathbb{R}^{m}}\dfrac{1}{\omega_{\Omega_2}(y)}\int_{\mathbb{R}^{n}}|F_mf(x,y)|dx dy\nonumber  \\
			&=& C(\Omega_1) \int_{\mathbb{R}^{n}}\int_{\mathbb{R}^{m}}\dfrac{|F_mf(x,y)|}{\omega_{\Omega_2}(y)}dy dx  \nonumber \\
			&\leq& C(\Omega_1)C(\Omega_2) \int_{\mathbb{R}^{n}}\int_{\mathbb{R}^{m}}| f(x,y)|dy dx . \nonumber 
		\end{eqnarray}
		Therefore we have proven that $C(\Omega_1\times\Omega_2)\leq C(\Omega_1)C(\Omega_2)$. For the other inequality, let $f\in \PW^1(2\Omega_1)$, $g\in \PW^1(2\Omega_2)$ and $h(x,y)=f(x)g(y)$. Then, since $\widehat{h}(x,y)=\widehat{f}(x)\widehat{g}(y)$ we have that $h\in \PW^1(2(\Omega_1\times\Omega_2))$ and thus 
		\begin{eqnarray}
			C(\Omega_1\times\Omega_2)&\geq& \frac{1}{\|h\|_{L^1}}	\int_{\mathbb{R}^{n+m}}\dfrac{|\widehat{h}(x,y)|}{\omega_{\Omega_1\times\Omega_2}(x,y)}dxdy \nonumber \\
			&=& \frac{1}{\|f\|_{L^1}\|g\|_{L^1}}	\int_{\mathbb{R}^{n}}\dfrac{|\widehat{f}(x)|}{\omega_{\Omega_1}(x)}dx \int_{\mathbb{R}^{m}}\dfrac{|\widehat{g}(y)|}{\omega_{\Omega_2}(y)}dy. \nonumber 
		\end{eqnarray}
		The inequality follows by taking supremum for all such $f,g$.
	\end{proof}
	It is evident that the above lemma also holds for a finite product of sets satisfying Hardy's inequality, thus we derive that Hardy's inequality holds for $\mathbb{R}^{n}_{+}$ and $I^{n}=I\times I\times...\times I$. A simple observation shows that Hardy's inequality also holds for $T(\mathbb{R}^{n}_{+})$ and $T(I)$ for all $T$ affine automorphisms of $\mathbb{R}^{n}$ for the same constants as Hardy's inequality for $\mathbb{R}_{+}^n$ and $I^n$ respectively. We will use this fact to prove Theorem \ref{poly} for simple polytopes in $\mathbb{R}^n$, i.e. polytopes such that every vertex belongs to exactly $n$ facets, by localizing near each vertex of the simple polytope. To achieve this, we will need the following technique.
	\begin{lm}\label{decomp}
		Let $\Omega$ be a bounded convex subset of $\mathbb{R}^{n}$ and suppose there exist open and bounded sets $A_1,..., A_N$ that cover $\overline{\Omega}$, such that for $\Omega\cap A_{j}$ Hardy's inequality holds for all $j=1,...,N$. Then Hardy's inequality holds for $\Omega$.
	\end{lm}
	\begin{proof}
		Since the sets $A_j$ are a finite open cover of $\overline{\Omega}$, we can find Schwarz functions $\phi_j$ such that $\supp\widehat{\phi}_{j}\subset 2A_j$ and $\sum_j |\widehat{\phi}_{j}|=1$ in $2\Omega$. 
		We notice that for $A\subset B$, $\omega_{A}\leq \omega_{B}$. Therefore, using Young's inequality we can compute
		\begin{eqnarray}
			\int_{\mathbb{R}^{n}} \frac{|\widehat{f}(x)|}{\omega_{\Omega}(x)}dx&=& \sum_{j=1}^N\int_{\mathbb{R}^{n}} \frac{|\widehat{f}(x)\widehat{\phi}_{j}(x)|}{\omega_{\Omega}(x)}dx\leq \sum_{j=1}^N\int_{\mathbb{R}^{n}} \frac{|\widehat{f}(x)\widehat{\phi}_{j}(x)|}{\omega_{\Omega\cap A_j}(x)}dx \nonumber \\
			&\lesssim&  \sum_{j=1}^N \|f\ast\phi_j\|_{L^{1}} \lesssim \|f\|_{L^{1}}, \nonumber  
		\end{eqnarray}
	\end{proof}
	For simple polytopes, we can show a uniform behavior of Hardy's inequality depending on the number of vertices.
	\begin{lm}\label{polyt}
		Let $P$ be a simple polytope in $\mathbb{R}^{n}$ of $N$ vertices. Then there exists a constant $C(n,N)>0$ such that
		$$\int_{\mathbb{R}^{n}}\dfrac{|\widehat{f}(x)|}{\omega_{P}(x)}dx\leq C(n,N)\|f\|_{L^{1}},$$ for all $f\in\PW^1(2P)$.
	\end{lm}
\begin{proof}
	Let us fix $\phi$ a Schwarz function such that $\widehat{\phi}(x)=0$ outside $C_n=(-1,1)^{n}$, $\widehat{\phi}(x)=1$ in $I^n=(-\frac{1}{2},\frac{1}{2})^{n}$ and $0\leq \widehat{\phi}(x)\leq 1$. Let $C(n,N)>0$ and $T_j$ affine maps, $j=1,...,C(n,N)$ such that $P\subset \cup_j T_j(I^n)$ and $P_j\cap P$ are parallilepipeds, where $P_j=T_j(C_n)$. To see this, it is obvious that it can be done for the cube, and thus by affine transformations, the number of which depends on the number of vertices, we can transfer it to $P$. Let us also define  $\widehat{\phi}_j(x)=\widehat{\phi}\circ (2T_j)^{-1}$. The result follows as in the proof of Lemma \ref{decomp} using the fact that $\|\phi_j\|_{L^{1}}=\|\phi\|_{L^{1}}$. Since by construction $\sum_{j=1}^{C(n,N)}|\widehat{\phi}_j(x)|\geq 1$ for all $x\in 2P$, 
	$$\int_{\mathbb{R}^{n}}\dfrac{|\widehat{f}(x)|}{\omega_{P}(x)}dx\leq \sum_{j=1}^{C(n,N)}\int_{\mathbb{R}^{n}}\dfrac{|\widehat{f}(x)\widehat{\phi}_j(x)|}{\omega_{P}(x)}dx\leq \sum_{j=1}^{C(n,N)}\int_{\mathbb{R}^{n}}\dfrac{|\widehat{f}(x)\widehat{\phi}_j(x)|}{\omega_{P\cap P_j}(x)}dx.$$
	By Lemmas \ref{interval} and \ref{prod} Hardy's inequality holds for all parallilepipeds for the same constant. Since by construction $P\cap P_j$ are parallilepipeds, we get that
	$$\int_{\mathbb{R}^{n}}\dfrac{|\widehat{f}(x)|}{\omega_{P}(x)}dx\lesssim \sum_{j=1}^{C(n,N)} \|f\ast\phi_j\|_{L^{1}}\lesssim \sum_{j=1}^{C(n,N)} \|f\|_{L^1}\|\phi_j\|_{L^{1}} \approx C(n,N)\|\phi\|_{L^{1}} \|f\|_{L^1}.$$
The proof is complete.
\end{proof}
The following lemma provides a way to extend our results from simple polytopes to general polytopes.
\begin{lm}\label{approxim}
	Let $\Omega$ be a bounded convex subset of $\mathbb{R}^{n}$ and suppose that we can find an increasing family of convex sets $P_j\subset \Omega$, $j\geq 0$ such that $\cup_j P_j=\Omega$ and Hardy's inequality holds for every $P_j$ for some uniformly bounded constants $C(P_j)$. Then Hardy's inequality holds for $\Omega$ for the constant $C=\liminf_{j\to \infty}C(P_j)$.
\end{lm}
\begin{proof}
	Let us set for $f\in\PW^1(2\Omega)$, $\widehat{f}_r(x)=\widehat{f}(\frac{x-2(1-r)z}{r})$, where $z\in \Omega$ and $r\in (0,1)$. Then $f_r\in \PW^1(2(r\Omega+(1-r)z))$. By convexity boundedness and monotonicity, for every $r\in (0,1)$ we can find $j_r\geq 0$ such that $r\Omega+(1-r)z\subset P_{j_r}$. For $A\subset B$, it is evident that $\omega_{A}\leq \omega_{B}$, and thus by Fatou's lemma and the fact that $\|f_r\|_{L^1}=\|f\|_{L^1}$,
	\begin{eqnarray} \int_{\mathbb{R}^{n}}\dfrac{|\widehat{f}(x)|}{\omega_{\Omega}(x)}dx&\leq&\liminf_{r\to 1}\int_{\mathbb{R}^{n}}\dfrac{|\widehat{f}_r(x)|}{\omega_{\Omega}(x)}dx\leq\liminf_{r\to 1}\int_{\mathbb{R}^{n}}\dfrac{|\widehat{f}_r(x)|}{\omega_{P_{j_r}}(x)}dx \nonumber \\ &\leq &\liminf_{r\to 1} C(P_{j_r})\|f_r\|_{L^1}= \liminf_{r\to 1} C(P_{j_r})\|f\|_{L^1}. \nonumber 
		\end{eqnarray} By the monotonicity of $P_j$, we can choose $P_{j_r}$ such that $\liminf_{r\to 1} C(P_{j_r})=\liminf_{j\to \infty} C(P_{j})$.
\end{proof}
	 Now we are finally able to present the complete proof of Theorem \ref{poly}. \\
	\\ \medspace 
	\textbf{Proof of Theorem \ref{poly}.} First, we notice that Hardy's inequality always implies Nehari's theorem for $p=2$ by Lemma \ref{neh2}. Thus it suffices to prove the validity of Hardy's inequality for any polytope. Let us fix a polytope $P$ in $\mathbb{R}^{n}$. By Lemmas \ref{polyt} and \ref{approxim} it suffices to find an increasing family of simple polytopes $P_j$ with uniformly bounded number of vertices such that $\cup_j P_j=P$. Instead of looking at this problem we are going to look at its polar. For an open convex set $\Omega$, let $\Omega^{\ast}$ be its polar body defined as, $$\Omega^{\ast}=\{y\in\mathbb{R}^{n}:\langle y,x\rangle < 1 \text{ for all }x\in \Omega\}.$$ Some important properties of the polar body are the following:
	\begin{enumerate}
		\item For $A$ open and convex with $0\in A$, $(A^{\ast})^{\ast}=A$.
		\item For $A,B$ open and convex with $A\subset B$, $B^{\ast}\subset A^{\ast}$.
		\item For $I$ an index set and $A_i$, $i\in I$ open and convex, $(\cup_{i\in I}A_i)^{\ast}=(\cap_{i\in I}A_i^{\ast})^{\circ}$, where $A^\circ$ denotes the interior of a set.
		\item The polar set of a polytope is a polytope.
		\item A polytope that contains $0$ is simple if and only if its polar is simplicial, that is all of its facets are the convex hull of $n$ affinely independent vectors. 
		\item If $P\subset\mathbb{R}^{n}$ is a polytope that contains $0$, then there is a bijection between the $l$ dimensional faces of $P$ and the $n-l-1$ dimensional faces of $P^{\ast}$, $l=0,1,...,n-1$.
	\end{enumerate} 
    For information about polar bodies the reader can refer to \cite[Chapter 14]{MR2335496}. Thus, it suffices to prove that for every polytope $P$ containing $0$, there exists a decreasing sequence of simplicial polytopes $P_j$ with uniformly bounded number of facets, such that $\cap_j P_j = P$. Rather than requiring a uniformly bounded number of facets, we will construct $P_j$ with a uniformly bounded number of vertices, which implies a uniform bound on the number of facets. Let us now fix a polytope $P$ that contains $0$. For a vertex $x$ of $P$, let us set $S(x)=\{\lambda(y-x):\lambda>0,y\in P\}$. $S(x)$ is called the support cone of $P$ at $x$. Let $x_1,...,x_k$ be the vertices of $P$. 
    \\ \textbf{Step 1.}
     For every choice of points $y_{x_i}\in x_i-S(x_i)$, $i=1,...,k$, close enough to $x_i$, we will show that their convex hull strictly contains $P$, that is their boundaries are disjoint. Let $\lambda_i\in (0,\frac{1}{k})$, $y_i\in P$ such that $y_{x_i}=x_i-\lambda_i(y_i-x_i)$ and let, since $y_i\in P$, $\mu_{i,j}\in (0,1)$ such that $y_i=\sum_{j=1}^k\mu_{i,j}x_j,$ $\sum_{j=1}^k\mu_{i,j}=1$. We want to show that every $x_i$ belongs to the interior of the convex hull of $y_{x_i}$. Without loss of generality we only show the case of $x_1$. We compute,
    \begin{eqnarray}
    	\sum_{i=1}^{k}\rho_i y_{x_i}&=&\sum_{i=1}^{k}\rho_i ((1+\lambda_i)x_i-\lambda_i y_i)=\sum_{i=1}^{k}\rho_i ((1+\lambda_i)x_i- \sum_{j=1}^k\lambda_i\mu_{i,j}x_j) \nonumber \\
    	&=& \sum_{i=1}^{k}\rho_i (1+\lambda_i)x_i - \sum_{i=1}^k\sum_{j=1}^k\rho_i\lambda_i\mu_{i,j}x_j \nonumber \\ &=& \sum_{i=1}^{k}\left(\rho_i (1+\lambda_i)-\sum_{j=1}^k\rho_j\lambda_j\mu_{j,i} \right)x_i. \nonumber 
    \end{eqnarray}
Let $\Lambda$ be the diagonal matrix $\Lambda=\diag(\lambda_i)_{1\leq i\leq k}$, $M=(\mu_{i,j})_{1\leq i,j\leq k}$, $\mathrm{P}=(\rho_i)_{1\leq i\leq k}$ be a column matrix and $E$ be the column matrix with first entry $1$ and zero elsewhere. We would like to have $x_1=\sum_{i=1}^{k}\rho_i y_{x_i}$ for some $\rho_i > 0$ with $\sum_{i=1}^{k}\rho_i =1$.
We will show that this can be achieved pointwise, i.e.
\begin{equation}\label{matrix} \rho_i (1+\lambda_i)-\sum_{j=1}^k\rho_j\lambda_j\mu_{j,i}=\delta_{i,1}.
\end{equation}
 This can be written in matrix language as
$$(I+\Lambda-M^{T}\Lambda)\mathrm{P}=E,$$ 
where $I$ is the identity $k\times k$ matrix.
First we observe that this matrix is invertible since,
$$\|\Lambda-M^{T}\Lambda\|\leq \|\Lambda\|\|I-M^{T}\|\leq k\max_{1\leq i\leq k}\lambda_i<1,$$ where the inequality $\|I-M^{T}\|\leq k$ holds since for any matrix $A=(a_{i,j})_{1\leq i,j\leq k}$, $\|A\|\leq k\max_{i,j}|a_{i,j}|$.
Now it suffices to show that the unique solution 
$$\mathrm{P}=(I+\Lambda-M^{T}\Lambda)^{-1}E,$$ has positive entries that sum to $1$. The fact that it sums to $1$ follows by adding equation \eqref{matrix} for all $i$, giving $$\sum_{i=1}^{k}\rho_i=\sum_{i=1}^{k}\left(\rho_i (1+\lambda_i)-\sum_{j=1}^k\rho_j\lambda_j\mu_{j,i} \right)=1.$$ 
Now for the positivity, let $I$ be the set of indices such that $\rho_i \geq 0$. Since $\sum_{i=1}^k\rho_i=1$, $I\neq\emptyset$. Summing again equation \eqref{matrix} for all $i\in I$ we get that
\begin{equation}\label{negsum} \sum_{i\in I}\left(\rho_i (1+\lambda_i)-\sum_{j=1}^k\rho_j\lambda_j\mu_{j,i} \right)=\delta,
	\end{equation} where $\delta=1$ if $1\in I$ and $0$ otherwise. We can observe that
\begin{align*}
	 \sum_{i\in I}\rho_i \lambda_i-\sum_{j=1}^k&\rho_j\lambda_j(\sum_{i\in I}\mu_{j,i}) 
	\\ &= \sum_{i\in I}\rho_i \lambda_i-\sum_{j\in I}\rho_j\lambda_j(\sum_{i\in I}\mu_{j,i})-\sum_{j\notin I}\rho_j\lambda_j(\sum_{i\in I}\mu_{j,i})  \\
	&= \sum_{i\in I}\rho_i \lambda_i(1-\sum_{j\in I}\mu_{i,j})+\sum_{j\notin I}(-\rho_j)\lambda_j(\sum_{i\in I}\mu_{j,i})\geq 0,  
\end{align*}
and thus, equation \eqref{negsum} gives that 
$$\delta\geq \sum_{i\in I}\rho_i\geq \sum_{i=1}^k \rho_i =1.$$ Therefore, $\delta=1$ and $\sum_{i\in I}\rho_i= \sum_{i=1}^k \rho_i$, which implies that $\rho_i\geq 0$ for all $i=1,...,k$. Now suppose that for some $i$, $\rho_i=0.$ Then by equation \eqref{matrix}, $$\delta_{i,1}=-\sum_{j=1}^k \rho_j \lambda_j \mu_{j,i}\leq 0.$$ This implies that $\rho_j=0$ for all $j=1,...,k$, contradiction.
\\ \textbf{Step 2.} For every $\epsilon>0$, we can find vectors $y_{x_i}\in x_i-S(x_i)$ with $|y_{x_i}-x_i|<\epsilon$, $i=1,...,k$, such that every collection of $n+1$ of them are affinely independent. We will use induction. Suppose that we have defined $y_{x_1},...,y_{x_l}$, such that each $n+1$ of them are affinely independent, and the distance of $y_{x_i}$ is less than $\epsilon$ from $x_i$. Define $\mathscr{F}$ as the union of affine hulls of all subsets of $n$ vectors chosen from $\{y_{x_1},..., y_{x_l}\}$. Then this is a finite union of at most $n-1$ dimensional sets, thus of measure zero. Since the set $(x_{l+1}-S(x_{l+1}))\cap B(x_{l+1},\epsilon)$ is open, it has positive measure, thus there must be an element that belongs to $(x_{l+1}-S(x_{l+1}))\cap B(x_{l+1},\epsilon)$ but not in $\mathscr{F}$. We choose this as $y_{l+1}$.
\\ \textbf{Step 3.} Let $\epsilon_0>0$ such that $B(0,\epsilon_0)\cap S(x_i)\subset \{\lambda(y-x_i):\lambda\in (0,\frac{1}{k}),y\in P\}$ for all $i=1,...,k$. Thus, for $\epsilon<\epsilon_0$, by steps $1$ and $2$, we can find $y_{x_i}\in x_i-S(x_i)$, $i=1,...,k$ such that $|y_{x_i}-x_i|<\epsilon$, every collection of $n+1$ of them are affinely independent and the convex hull $Q_{\epsilon}=\conv\{y_{x_i},i=1,...,k\}$ strictly contains $P$. $Q_{\epsilon}$ by definition is simplicial. Since $|y_{x_i}-x_i|<\epsilon$, $y_{x_i}\in P+B(0,\epsilon)$, and thus $Q_{\epsilon}\subset P+B(0,\epsilon)$.
Let $\epsilon_1<\epsilon_0$. Since $P$ is strictly contained in $Q_{\epsilon_1}$, the distance of their boundaries is strictly positive, thus for $\epsilon_2<\min(\dist(\partial P,\partial Q_{\epsilon_1}),\frac{\epsilon_0}{2})$ we have that $Q_{\epsilon_2}\subset P+B(0,\epsilon_2)\subset Q_{\epsilon_1}$. By induction we can construct $\epsilon_j<\frac{\epsilon_0}{j}$, such that $Q_{\epsilon_j}$ satisfy $$P\subset\bigcap_{j\in\mathbb{Z}_{+}}Q_{\epsilon_j}\subset \bigcap_{j\in\mathbb{Z}_{+}} (P+B(0,\epsilon_j))\subset \bigcap_{j\in\mathbb{Z}_{+}} (P+B(0,\epsilon_0/j))=P.$$
The proof is complete.
 \qed	
    \\ \newline \par In three dimensions, the argument simplifies significantly. In $\mathbb{R}^3$, every edge of a polytope lies in exactly two facets. Therefore, given a polytope $P$ and one of its vertices $x$, we can truncate $x$ using a hyperplane near it, resulting in a fixed number of new simple vertices. By performing this operation at all vertices, we obtain a simple polytope. Moreover, by choosing hyperplanes that cut closer and closer to the original vertices, we generate a family of simple polytopes with uniformly bounded number of vertices that approximates $P$.
    
    As an example, let $P = T(\alpha, \beta)$ be the pyramid in $\mathbb{R}^3$ defined in Section \ref{sec2}. If we truncate the apex $(0,0,\beta)$ using the horizontal hyperplane $\{(x, y, z) : z = z_0\}$, we obtain the truncated set
    $$
    Q_{z_0} = \left\{ (x, y, z) \in \mathbb{R}^3 : |x|, |y| < \alpha - \frac{\alpha}{\beta}z, \; 0 < z < z_0 \right\}.
    $$
    This set retains the original base vertices $(\pm\alpha, \pm\alpha, 0)$, and replaces the apex $(0, 0, \beta)$ with four new simple vertices:
    $$
    \left(\pm\left(\alpha - \frac{\alpha}{\beta}z_0\right), \pm\left(\alpha - \frac{\alpha}{\beta}z_0\right), z_0\right).
    $$
    Hence, $Q_{\beta-1/j}$ forms a family of polytopes with only simple vertices. In this case, the original polytope had a single non-simple vertex; if more were present, we would perform additional truncations accordingly.
    \par In the proof of Theorem \ref{poly}, we showed that in the polar case, the approximating polytopes have the same number of vertices as the original polytope. Since polarity interchanges vertices and facets, by Lemmas \ref{polyt} and \ref{approxim} we can conclude that for any polytope $P\subset\mathbb{R}^{n}$ of $N$ facets, there exists a constant $C(n,N)>0$ such that 
    $$\int_{\mathbb{R}^n}\dfrac{|\widehat{f}(x)|}{\omega_{P}(x)}dx\leq C(n,N)\|f\|_{L^1},$$ for all $f\in \PW^1(2P)$.
	
	\section{Adjusted Hardy's inequality}\label{sec4}
	Let us consider the following family of inequalities, which are Hardy's inequality with a power adjustment on the weight, that is, for $d\in\mathbb{R}$,
	\begin{eqnarray}\label{11} \int_{\mathbb{R}^{n}}\dfrac{|\widehat{f}(x)|}{\omega_{\Omega}^d(x)}dx\leq C\|f\|_{L^{1}}, \quad f\in\PW^{1}(2\Omega).
	\end{eqnarray} 
	Theorem \ref{poly} says that inequality \eqref{11} holds for $d=1$, and thus for all $d\leq 1$, for every polytope. For some $d$ small enough (depending on the dimension), when $\Omega$ is bounded we can easily prove that such an inequality holds, using the following result of Schmuckenschl\"{a}ger \cite[Theorem 2]{MR1194970} which describes the asymptotic behavior of the measure of the sets $\{x\in 2\Omega:\omega_{\Omega}(x)<t\}$.
	\begin{lm}[Schmuckenschl\"{a}ger]\label{schm} For a bounded convex $\Omega\subset\mathbb{R}^{n}$, there is $C(\Omega)\geq 0$ such that 
     $$\lim_{t\to 0}\dfrac{m(\{x\in 2\Omega:\omega_{\Omega}(x)<t\})}{t^{\frac{2}{n+1}}}=C_n.$$
	\end{lm}
	Since the function $$f(t)=\dfrac{m(\{x\in 2\Omega:\omega_{\Omega}(x)<t\})}{t^{\frac{2}{n+1}}},$$ is continuous in $t\in (0,\infty)$ with $\lim_{t\to 0}f(t)<\infty$ and $\lim_{t\to \infty}f(t)=0$, we get that there is $K(\Omega)>0$ such that $m(\{x\in 2\Omega:\omega_{\Omega}(x)<t\})\leq K(\Omega)t^{\frac{2}{n+1}}$ for all $t>0$. Thus we have the following lemma.
	\begin{lm}\label{hold}
		Let $\Omega$ be a convex bounded set. Then, inequality \eqref{11} holds for all $d<\frac{2}{n+1}$.
	\end{lm}
	\begin{proof}
		If $\Omega$ is a polytope then it follows by Theorem \ref{poly}. We will prove that $\omega_{\Omega}^{-d}\in L^{1}(2\Omega)$ whenever $d<\frac{2}{n+1}$. Since $\omega_{\Omega}$ is bounded by $m(\Omega)$, we can assume that $d>0$. Let, without loss of generality $\omega_{\Omega}(x)\leq 1$. By Schmuckenschl\"{a}ger's theorem (Lemma \ref{schm}) we can compute
			$$\int_{2\Omega}\omega_{\Omega}^{-d}(x)dx= d\int_{1}^{\infty}t^{d-1}m(\{y\in 2\Omega:\omega_{\Omega}^{-1}(y)>t\})dt \leq K(\Omega)\int_{1}^{\infty}t^{d-1-\frac{2}{n+1}}dt,$$
	which converges whenever $d<\frac{2}{n+1}$. Thus, for this case, we have that 
		$$\int_{\mathbb{R}^n}\dfrac{|\widehat{f}(x)|}{\omega_{\Omega}^d(x)}dx\leq \|\widehat{f}\|_{L^{\infty}}\int_{2\Omega}\omega_{\Omega}^{-d}(x)dx\leq \|f\|_{L^{1}}\int_{2\Omega}\omega_{\Omega}^{-d}(x)dx.$$ The proof is complete.
	\end{proof}
	Finally, we turn our attention to the negation of inequality \eqref{11} and we give the following counter argument which depends on the existence of homogeneous sets. The following lemma is Theorem \ref{theo3} (3).
\begin{lm}\label{John}
	Let $\Omega$ be a convex subset of $\mathbb{R}^{n}$ free of lines. Then Hardy's inequality fails whenever $d>1$, and if furthermore $\Omega$ is unbounded, then it fails whenever $d\neq1$.
\end{lm}
\begin{proof}
	Let $t\in\omega_{\Omega}(2\Omega)$, and let $x_t\in2\Omega$ such that $\omega_{\Omega}(x_t)=t$. Since $\Omega\cap (x_t-\Omega)$ is a convex set of measure $t$, by John's Theorem \cite{John1948} we can find an ellipse $A_t\subset \Omega\cap (x_t-\Omega)$ such that $m(A_t)\approx t$. Also, for $y\in A_t$, there is $z\in \Omega$ with $x_t=y+z$, and thus by concavity of $\omega_{\Omega}^{\frac{1}{n}}$ (\cite[Lemma 2.1]{bampouras2024besovspacesschattenclass}), $\omega_{\Omega}(y)\leq 2^{n}\omega_{\Omega}(x_t) =2^{n}t$. Let us consider for our counterexample, $\widehat{\phi}$ Schwarz function supported in the unit ball $B(0,1)$, $\widehat{\phi}=1$ in $\frac{1}{2} B(0,1)$ and $\widehat{\phi}_t=\widehat{\phi}\circ T_t$, where $T_t$ is the affine transform that maps $A_t$ onto $B(0,1)$. Applying $\phi_t$ on Hardy's inequality, since $\|\phi_t\|_{L^1}= \|\phi\|_{L^1}$,
	$$t^{1-d}\lesssim 1.$$
	This completes the proof.
\end{proof}
Theorem \ref{theo3} is the combination of Theorem \ref{poly} and Lemmas \ref{hold} and \ref{John}.

	\bibliographystyle{plain}
	\bibliography{Bampouras3}
	
\end{document}